\documentclass[letterpaper,11pt]{article}

\usepackage{ucs}
\usepackage[utf8x]{inputenc}
\usepackage{graphicx}
\usepackage{amsfonts}
\usepackage{dsfont}
\usepackage{amssymb}
\usepackage{amsmath}
\usepackage{amsthm}
\usepackage{enumerate}
\usepackage{stmaryrd}
\usepackage{fullpage}
\usepackage{ifthen}
\usepackage{subfigure}
\usepackage{epic}
\usepackage{authblk}
\usepackage{textcomp}
\usepackage[small]{caption}
\usepackage{bm}

\usepackage{color}
\usepackage{titlesec}
\graphicspath{{eps/}{pdf/}}

\newcommand{\be}{\begin{equation}}
\newcommand{\ee}{\end{equation}}
\newcommand{\bqs}{\begin{equation*}}
\newcommand{\eqs}{\end{equation*}}

\newcommand{\rrrrr}{r_1}
\newcommand{\TTTTT}{T_1}
\newcommand{\WWWWW}{W_1}

\newcommand{\e}{\varepsilon}

\renewcommand{\O}{\mathcal{O}}

\numberwithin{equation}{section}

  {\begin{trivlist}\item[]{\bf Hypothesis #1 }\em}{\end{trivlist}}

\theoremstyle{plain}
\newtheorem{theorem}{Theorem}[section]

\newtheorem{lemma}[theorem]{Lemma}

\newtheorem{rmk}[theorem]{Remark}

 {\begin{trivlist}\item[]\textbf{Proof#1 }}%
 {\hspace*{\fill}$\rule{0.3\baselineskip}{0.35\baselineskip}$\end{trivlist}}

\title{Pushed fronts in a Fisher-KPP-Burgers system using geometric desingularization}

\author{Matt Holzer, Matthew Kearney, Samuel Molseed, Katie Tuttle and David Wigginton}
\affil{\small Department of Mathematical Sciences, George Mason University, Fairfax, VA, USA}

\begin{document}
\maketitle

\begin{abstract} We study traveling fronts in a system of reaction-diffusion-advection equations in one spatial dimension motivated by problems in reactive flows.  In the limit as a parameter tends to infinity, we construct the approximate front profile and determine the leading order expansion for the selected wavespeed.  Such fronts are often constructed as transverse intersections of stable and unstable manifolds of the traveling wave differential equation.  However, a re-scaling of the dependent variable leads to a lack of hyperbolicity for one of the end states making the definition of one such manifold unclear.  We use geometric blow-up techniques to recover hyperbolicity and following an analysis of the blown up vector field are able to show the existence of a traveling front with a leading order expansion of its speed.  

\end{abstract}



\section{Introduction}
The following system of reaction-diffusion-advection equations was recently introduced in \cite{bramburger21},
\begin{eqnarray}
T_t &=& T_{xx}-(uT)_x+T(1-T) \nonumber \\
u_t &=& \nu u_{xx} -uu_x+\rho T(1-T). \label{eq:RDA}
\end{eqnarray}
Here $T(t,x)$ models the temperature of a fluid and $u(t,x)$ represents the velocity of that fluid.  The reaction term in the $T$ component introduces growth of the temperature field due to the occurrence of a chemical reaction.  The temperature increase also induces growth in the velocity profile with some proportionality constant $\rho$.  Both temperature and velocity also are influenced by diffusion and advection.  System (\ref{eq:RDA}) gives rise to traveling front solutions.  Motivated by similar systems of coupled fluid and reaction-diffusion equations such as \cite{constantin08,malham98} the goal of \cite{bramburger21} is to understand how coupling between the two variables affects the speed of these fronts.  A more detailed description of model (\ref{eq:RDA}) and its relation to other models of reactive flow can be found in \cite{bramburger21}. 

\paragraph{Traveling wave equation} The authors in \cite{bramburger21} study traveling fronts for (\ref{eq:RDA}) with a particular interest in how the speeds of these fronts depend on the parameters $\rho$ and $\nu$.  The purpose of this paper is to study fronts in the inviscid problem $(\nu=0)$ and in the limit as $\rho\to\infty$.  When $\nu=0$ the traveling wave equations for (\ref{eq:RDA}) form a system of three first-order ordinary differential equations,
\begin{eqnarray}
T'&=& -\tilde{c} T +UT +V \nonumber \\
U'&=& \frac{\rho}{U-\tilde{c}}T(1-T) \nonumber \\
V' &=& T(T-1). \label{eq:TWnotreduced} 
\end{eqnarray}
This system has a conserved quantity; $\frac{1}{2}U^2-\tilde{c}U +\rho V$ which after noting that we are interested in solutions converging to zero can be used to reduce (\ref{eq:TWnotreduced}) a planar system;  see \cite{bramburger21}
\begin{eqnarray}
   T' &=& -\tilde{c}T+UT+\frac{1}{2\rho}U(2\tilde{c}-U) \nonumber\\
    U' &=& \frac{\rho}{U-\tilde{c}}T(1-T). \label{eq:mainTW} 
\end{eqnarray}
\paragraph{Front and wavespeed selection} 
Traveling front solutions to (\ref{eq:RDA}) can be found by locating heteroclinic orbits of (\ref{eq:mainTW}).   We are interested in fronts connecting the state $(1,\tilde{c}+\rho-\sqrt{\tilde{c}^2+\rho^2})$ to the zero state $(0,0)$.  Since the zero state is unstable as a solution of (\ref{eq:RDA}) this is a problem in the study of fronts propagating into unstable states; see for example \cite{vansaarloos03}.  These fronts and the corresponding heteroclinic orbits exist for a continuum of wavespeeds.  This is most easily observed in (\ref{eq:mainTW}) after noting that $(0,0)$ is a stable node while $(1,\tilde{c}+\rho-\sqrt{\tilde{c}^2+\rho^2})$ is a saddle.   Therefore, if a heteroclinic (equivalently front) exists for some wavespeed then it must also exist for an open set of nearby speeds.  Among this family of fronts one is typically interested in identifying the {\em selected} or {\em critical} front (its speed is called the {\em selected wavespeed}).  This front is the one that attracts compactly supported initial conditions of the PDE (\ref{eq:RDA}) \footnote{Such a front does not have to be unique.  A simple example is Nagumo's equation where different speeds can be observed for non-negative and non-positive initial data. However, even if one restricts to positive, but compactly supported initial data there can be multiple fronts which can be selected by compactly supported initial data; see for example \cite{faye19}.}.  Critical fronts are defined and identified in the literature using several different criterion.  In different settings the critical front is defined as the slowest positive front, the slowest monotone front, the marginally stable front, or directly as the front whose basin of attraction includes compactly supported initial conditions.  Although not generally the case, it happens to be true that for many systems of PDEs, these different criterion all identify the same traveling front solution.

\paragraph{Pushed and pulled fronts} Typically the critical front can be classified as being of one of two types: pulled or pushed.  Pulled fronts are driven by the instability ahead of the front interface and their spreading speed equals that of compactly supported initial conditions for the equation linearized near the unstable steady state.  Pushed fronts are driven by nonlinear effects behind the front interface and propagate at (typically, see \cite{holzer12}) faster-than-linear speeds.  As we mentioned above, determining the front selected by compactly supported initial data can be associated to stability properties of the fronts; see \cite{avery22,dee83,vansaarloos03}.  Pulled fronts have essential spectrum that touches the imaginary axis in an optimally chosen exponentially weighted function space while pushed fronts have a zero translational eigenvalue and stable essential spectrum in the weighted space.  This translational eigenvalue has an eigenfunction given by the derivative of the front profile and so for this function to be in the weighted space requires strong exponential decay of the front.    We refer the interested reader to \cite{vansaarloos03} for an in-depth discussion of pushed and pulled fronts and we note recent work on pushed and pulled fronts and the transition between them; see \cite{an21,avery23}.    

 Bringing the discussion back to the traveling wave equation (\ref{eq:mainTW}), we note that restricting our search to traveling fronts with strong decay near the origin is equivalent to locating heteroclinic orbits involving intersections of the one-dimensional unstable manifold of the saddle state at $(1,\tilde{c}+\rho-\sqrt{\tilde{c}^2+\rho^2})$ and the one dimensional strong-stable manifold of the origin.  Such a heteroclinic connection is not structurally stable under perturbations and we expect to identify a unique speed for which such a connection occurs.

\paragraph{Constructing pushed fronts: singular perturbations} The explicit construction of heteroclinic orbits for nonlinear ordinary differential equations is difficult.  Most examples where pushed invasion speeds can be  calculated involve scalar reaction-diffusion equations where explicit solutions can be obtained as in the classical Nagumo's equation; see \cite{hadeler75}.  For systems of reaction-diffusion equations the situation is more challenging still, but one promising avenue is when one or more parameters in the system are asymptotically small (or large).  In these situations, methods from singular perturbation theory (see \cite{fenichel79,jones95}) can be employed to construct solutions and wavespeed estimates are sometimes possible by analyzing a reduced planar system; see for example \cite{harley14,holzer12,hosono03}. 

\paragraph{Previous work} Among other results, the following estimate of the critical speed in the limit as $\rho\to\infty$  was obtained in \cite{bramburger21}.  

\begin{theorem}\cite{bramburger21}\label{thm:bram} System (\ref{eq:RDA}) with $\nu=0$ has traveling front solutions connecting the stable state $(1,\tilde{c}+\rho-\sqrt{\tilde{c}^2+\rho^2})$ to $(0,0)$ for any $\tilde{c}\geq \tilde{c}^*(\rho)$.  These fronts are monotonic and for $\rho\to\infty$ the critical speed $\tilde{c}^*(\rho)$ satisfies 
\be \sqrt[3]{\frac{3}{2}} \leq \liminf_{\rho\to\infty} \frac{\tilde{c}^*(\rho)}{\rho^{1/3}}\leq \limsup_{\rho\to\infty} \frac{\tilde{c}^*(\rho)}{\rho^{1/3}} \leq \sqrt{3}  \label{eq:speedbramburger} \ee
\end{theorem}

Numerical estimates of the critical speed were obtained in \cite{bramburger21} using methods from \cite{bramburger20} and suggest that the critical speed should scale like the lower bound in (\ref{eq:speedbramburger}).  The purpose of the current research is to refine the estimate in Theorem~\ref{thm:bram} and show that the critical speed does in fact scale with $ \sqrt[3]{\frac{3\rho}{2}}$ as $\rho\to\infty$.   

\paragraph{Proof strategy} We make the following changes of coordinates: $\rho=\frac{1}{\e^3}$, $c=\tilde{c}/\e$, $W=\e U$ and re-scale the independent variable so that (\ref{eq:mainTW}) is transformed to the following system of equations with small parameter $\e$,
\begin{align}
    \dot{T} &= -cT+WT+\frac{1}{2}W\varepsilon^2(2c-W)\nonumber \\
    \dot{W} &=\frac{1}{W-c}T(1-T). \label{eq:scale1}
\end{align}

A general framework for proving the existence of traveling waves (i.e. heteroclinic orbits)  in systems such as (\ref{eq:scale1}) is as follows; see for example \cite{jones95}.  One first sets the small parameter to zero and identifies a heteroclinic orbit in this singular limit.  In the case of (\ref{eq:mainTW}) this heteroclinic involves the intersection of two one-dimensional manifolds  (namely a heteroclinic connecting the unstable manifold of $(1,\tilde{c}+\rho-\sqrt{\tilde{c}^2+\rho^2})$ and the strong-stable manifold of $(0,0)$).  As such, the system is not structurally stable and the heteroclinic is not expected to persist when $\e\neq 0$.  However, if the system has a parameter ($c$ in our case) then one can add the parameter as a variable to increase the dimension of the system and construct two dimensional center-stable and center-unstable manifolds.  If these manifolds intersect transversely and depend smoothly on $\e$ then the transverse intersection will persist for $\e>0$ and the existence of a heteroclinic follows without having to delve into the particular manner in which $\e$ appears in the equations.  

The primary complication occurring in (\ref{eq:scale1}) is that the heteroclinic in the singular limit involves a connection to a non-hyperbolic fixed point.  Thus, even the persistence of this fixed point when $\e\neq 0$ is not apparent and it is not possible to define stable or unstable manifolds of this fixed point without explicitly considering $\e$ dependent terms.  To overcome this issue we will apply geometric desingularization techniques to blow-up the non-hyperbolic fixed point and regain hyperbolicity.  We refer the reader to \cite{dumortier92,kuehn15} for a description of the method.  Related to the problem of front propagation: blow-up has been used to compute correction to wavespeeds due to cutoff in reaction terms; see \cite{dumortier07}. We also point to other examples where re-scalings of the dependent variables lead to a lack of hyperbolicity -- see for example \cite{carter18,gucwa09,holzer17}.

The main result of this paper is the following.  
\begin{theorem}\label{thm:main} Let $\nu=0$ in (\ref{eq:RDA}).  Then there exists a $\rho_0>0$ such that for any $\rho>\rho_0$ there exists a heteroclinic orbit for (\ref{eq:mainTW}) corresponding to a traveling front solution of (\ref{eq:RDA}) connecting the stable state $(1,\tilde{c}+\rho-\sqrt{\tilde{c}^2+\rho^2})$ to $(0,0)$ with speed $\tilde{c}^*(\rho)$ and the following properties.  
\begin{enumerate}
\item The heteroclinic orbit lies in the strong-stable manifold of $(0,0)$ and  the traveling front is monotone decreasing and has steep exponential decay in the sense that
\[  |(T^*(x),Q^*(x))|\leq C e^{-\tilde{c}^*(\rho)x/2} \ \text{as} \ x\to\infty. \]
\item The speed $\tilde{c}^*(\rho)$ satisfies $\lim_{\rho\to\infty} \frac{\tilde{c}^*(\rho)}{\sqrt[3]{\rho}}=\sqrt[3]{\frac{3}{2}}.$
\end{enumerate} 
\end{theorem}

The remainder of the paper is dedicated to proving this Theorem and is organized as follows.  In Section~\ref{sec:prelim} we change coordinates in the traveling wave equation (\ref{eq:scale1}) and collect some facts about this transformed system.  In Section~\ref{sec:hetero} we set $\e=0$ and obtain a candidate traveling front.   This front involves a heteroclinic orbit involving a non-hyperbolic fixed point.  In Section~\ref{sec:blowup} we apply geometric desingularization techniques to  track invariant manifolds near this point.  In Section~\ref{sec:proof} we conclude the proof of Theorem~\ref{thm:main}.  Finally, we present a short discussion in Section~\ref{sec:discussion}.

\section{Preliminaries}\label{sec:prelim}

We begin our analysis by removing the singularity that occurs in (\ref{eq:scale1}) at $W=c$.  This is accomplished by rescaling the independent variable so that the right hand side of (\ref{eq:scale1}) is multiplied by the non-zero factor $c-W$.  Since we are interested in the region where $W<c$ then this quantity is always positive and does not reverse the direction of the flow.  We then arrive at the system 
\begin{align}
    \dot{T} &= -T(W-c)^2-\frac{1}{2}W\varepsilon^2(2c-W)(W-c)\nonumber \\
    \dot{W} &=-T(1-T). \label{eq:TWmainscaling}
\end{align}
System (\ref{eq:TWmainscaling}) has fixed points at $(0,0)$, $(1,c)$ and $(1,w_\pm(\e))$.  The fixed point at $(1,c)$ is an artifact of the re-scaling to remove the singularity at $W=c$.  The fixed points $w_\pm(\e)$ are fixed points of the original system (\ref{eq:mainTW}) transformed to $(T,W)$ variables.  The $w_\pm(\e)$ are roots of the quadratic polynomial 
\[ h(w)=w-c+\frac{1}{2}\e^2 w (2c-w), \]
which can be expressed as 
\[ w_\pm(\e)=c+\frac{1}{\e^2}\pm \frac{1}{\e^2}\sqrt{1+\e^4c^2 }. \]
We are interested in fronts connecting $(1,w_-(\e))$ to $(0,0)$ and note that 
\be w_-(\e)=c-\e^2\frac{c^2}{2} +\O(\e^4). \label{eq:w-exp} \ee
It will be consequential later that $w_-(\e)<c$ and $w_-(\e)\to c$ as $\e\to 0$. 

The linearization of (\ref{eq:TWmainscaling}) at $(0,0)$ has two negative eigenvalues 
\be \mu_\pm(c,\e)=-\frac{c^2}{2}\pm\frac{c^2}{2}\sqrt{1+4\e^2}.\label{eq:mu} \ee
The linearization at $(1,w_-(\e))$ is 
\[ \left(\begin{array}{cc} -(w_-(\e)-c)^2 & -(w_-(\e)-c)h'(w_-(\e)) \\ 1 & 0 \end{array}\right). \] 
Since the trace is negative and the determinant is negative (note $h'(w_-(\e))>0$) we then see that the fixed point $(1,w_-(\e))$ is a saddle with one stable and one unstable eigenvalue.  Regrettably, both these eigenvalues are close to zero so that when $\e=0$ the fixed point is no longer hyperbolic and its linearization is nilpotent.  We emphasize that this lack of hyperbolicity can be traced to the fact that as $\e\to 0$ the fixed point at $(1,w_-(\e))$ coalesces with the fixed point at $(1,c)$ yielding a non-hyperbolic fixed point.  

Nonetheless, for any $0<\e\ll 1$ the unstable manifold $W^u(1,w_-(\e))$ is well defined and can be written as a graph $(T,h_u(c,\e,T))$.  Likewise, for $c$ fixed and $0<\e\ll 1$ the origin has a one dimensional strong stable manifold $W^{ss}(0,0)$ which can also be expressed as a graph $(T,h_s(c,\e,T))$.  We then define the mismatch function 
\be \Phi(c,\e)=h_s\left(c,\e,\frac{1}{2}\right) - h_u\left(c,\e,\frac{1}{2}\right).\label{eq:Phi}  \ee

We will establish the following result.

\begin{theorem}\label{thm:Phifacts} We have the following 
\begin{itemize}
\item The function $\Phi(c,\e)$ is well defined for $\e\in [0,\e_0)$, for some $\e_0>0$ and continuously differentiable in both $c$ and $\e$. 
\item $\Phi\left( \sqrt[3]{\frac{3}{2}},0\right)=0$ 
\item $\partial_c \Phi\left( \sqrt[3]{\frac{3}{2}},0\right)\neq 0 $
\end{itemize}
\end{theorem}

By an application of the Implicit Function Theorem,  Theorem~\ref{thm:Phifacts} will then imply Theorem~\ref{thm:main}.   We remark that an important piece of establishing Theorem~\ref{thm:Phifacts} is verifying that the function $\Phi(c,\e)$ is well defined in the limit as $\e\to 0$.  

\begin{rmk} In a broader context, we emphasize that pushed fronts are robust with respect to small perturbations in the system.  This is because pushed fronts can typically be expressed as a transverse intersection of manifolds.  These manifolds depend smoothly on parameters and the intersection therefore persists under small perturbations.  The conditions laid out in Theorem~\ref{thm:Phifacts} exactly describe this transverse intersection: condition (ii) in Theorem~\ref{thm:Phifacts} states that these manifolds intersect and condition (iii) implies that this intersection is transverse.  To emphasize once more, the issue with (\ref{eq:TWmainscaling}) is the lack of hyperbolicity of the fixed point $(1,c)$ making the definition of the unstable manifold of $(1,c)$ not well defined a priori.  
\end{rmk}

\section{Reduction to the singular limit  $\e=0$ } \label{sec:hetero}
In this section we study (\ref{eq:TWmainscaling}) with $\e=0$ and construct a heteroclinic orbit that connects the fixed point $(1,c)$ to the one at $(0,0)$.  With $\e=0$ the system reduces to 
\begin{align}
    \dot{T} &= -T(W-c)^2 \nonumber \\
    \dot{W} &=-T(1-T). \label{eq:reduced}
\end{align}
For $0<T<1$ solutions of this system are graphs with respect to the variable $T$ and their solution curves obey the scalar equation 
\begin{align}
    \frac{dW}{dT} = \frac{1-T}{(W-c)^2}.
\end{align}
This equation is separable and can be integrated to find solution curves that satisfy
\begin{equation} \frac{1}{3} (W-c)^3 = T - \frac{T^2}{2} + k, \label{eq:solutioncurves} \end{equation}
where $k$ is a constant of integration.  

For the curve in (\ref{eq:solutioncurves}) to pass through the origin it is required that $k_0=-\frac{c^3}{3}$.  For the curve to pass through the point $(1,c)$ it is required that $k_1=-\frac{1}{2}$.  Therefore, the curves intersect if and only if $c=c^*(0)$ with 
\be c^*(0)=  \sqrt[3]{\frac{3}{2}}. \ee
A plot of this solution curve is provided in Figure~\ref{fig:hetero}.

\begin{figure} 
    \centering
     \subfigure{\includegraphics[width=0.43\textwidth]{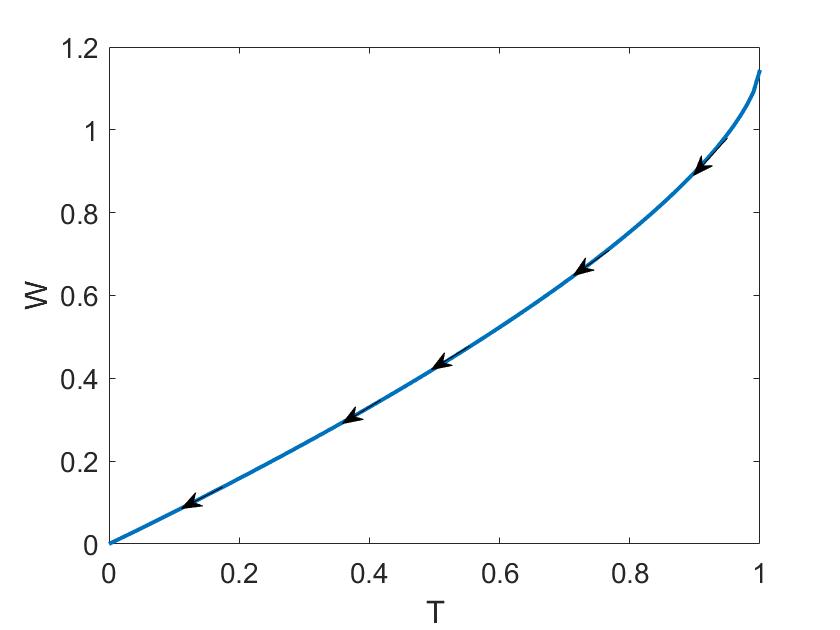}}
 \caption{The heteroclinic orbit obtained for the reduced ($\e=0$) system connecting the fixed point $(1,c)$ to the origin with $c=c^*(0)=\sqrt[3]{\frac{3}{2}}$. }
    \label{fig:hetero}
\end{figure}

For future reference we define the functions 
\[ g_s(c,T)=c-\sqrt[3]{-3T+\frac{3}{2}T^2+c^3},\]
which gives the $W$ component of the graph of the solution curve passing through the origin while we define
\[ g_u(c,T)=c-\sqrt[3]{-3T+\frac{3}{2}T^2+\frac{3}{2}}.\]

We make note of the following fact.

\begin{lemma}\label{lem:IFTcond}
For any $T>0$
\[ \frac{\partial}{\partial c} \left(g_s(c,T)-g_u(c,T)\right)\neq 0. \]
\end{lemma} 
\begin{proof}
We compute
\[ \frac{\partial}{\partial c} \left(g_s(c,T)-g_u(c,T)\right)=-\frac{c^2}{(-3T+\frac{3}{2}T^2+c^3)^{2/3}}\neq 0. \]
\end{proof}

\section{A quasi-homogeneous blow-up of the nilpotent fixed point $(1,c)$} \label{sec:blowup}
When $\e=0$ the fixed point $(1,c)$ is not hyperbolic.  To prove Theorem~\ref{thm:main} we need to show that the unstable manifold of $(1,w_-(\e))$ is $o(1)$ close to the heteroclinic orbit computed in Section~\ref{sec:hetero}.  To do this we will use geometric desingularization to blow-up and then desingularize the flow near this fixed point.  

To begin, we first transform the fixed point at $(1,c)$ to the origin using 
\[ \Tilde{W} = W - c, \quad \Tilde{T}=T-1, \]
from which we obtain the system of equations (appending an equation for  the parameter $\e$)
\begin{align}
  \dot{\Tilde{T}}' &= - \Tilde{W}^2 (\Tilde{T} + 1) -  \frac{1}{2}\Tilde{W}\varepsilon^2 (c^2 - \Tilde{W}^2) \nonumber \\
  \dot{\Tilde{W}} &= \Tilde{T}(\Tilde{T} + 1)  \nonumber \\
\dot{\e} &= 0. \label{eq:TWtildes} 
\end{align}
The linearization at the origin is nilpotent with Jacobian
\[ Df(0,0,0)=\left(\begin{array}{ccc} 0 & 0 & 0  \\ 1 & 0 & 0 \\ 0 & 0 & 0\end{array}\right).\]

The goal of this section is to blow-up the non-hyperbolic origin in (\ref{eq:TWtildes}) to a sphere where -- after desingularizing by a re-scaling of the independent variable -- the flow in a neighborhood of the origin can be analyzed by studying the  flow on the sphere.   In the simplest case, blow-up techniques involve viewing (\ref{eq:TWtildes})  in spherical coordinates $(\rho,\theta,\phi)$ where after desingularizing one obtains non-trivial dynamics on the surface of the sphere ($\rho=0$).   However, use of such coordinates for  (\ref{eq:TWtildes}) reveals fixed points on the surface of the sphere which remain non-hyperbolic.  

We therefore employ a quasi-homogeneous blow-up where different scalings are given to different variables; see \cite{alvarez11,dumortier92,jk21,kuehn15} for examples and a general discussion.  The change of variables that we employ is defined as
\[ \psi:\mathbb{S}^2\times [0,\infty) \to \mathbb{R}^3,\quad \psi\left(\bar{T},\bar{W},\bar{\e},r\right)=\left(r^3\bar{T},r^2\bar{W},r\bar{\e}\right),\]
where $(\bar{T},\bar{W},\bar{\e})\in\mathbb{S}^2$.  We comment on the choice of weights.  Since the fixed point $w_-(\e)$ depends quadratically on $\e$ then it is natural that the weight for the $\tilde{W}$ component should be squared relative to the weight for $\e$.   The cubic scaling for the $\tilde{T}$ component can be determined as follows.  Let $\tilde{T}=r^\alpha \bar{T}$. Then the leading order terms on the right hand side of (\ref{eq:TWtildes}) scale with $r^4$ for $\tilde{T}'$ equation and $r^\alpha$ for the $\tilde{W}'$ equation.  Thus the scaling of the two equations match only if $4-\alpha=\alpha-2$ from which $\alpha=3$ is determined.

The variables $\left(\bar{W},\bar{T},\bar{\e}\right)\in\mathbb{S}^2$ could be expressed in terms of spherical coordinates; however in practice it is easier to use various coordinate charts in which to view the dynamics.  We will employ the following two charts for our analysis
\begin{eqnarray*}
K_\e : &\ & \tilde{T}=\rrrrr^3 \TTTTT, \ \tilde{W}= -\rrrrr^2\WWWWW, \ \e=\rrrrr \\
K_W :&\ & \tilde{T}=r_2^3 T_2, \ \tilde{W}= -r_2^2, \ \e=r_2\e_2,
\end{eqnarray*}
where we will assume $r_{1,2}\geq 0$, $W_1\geq 0$, $T_{1,2}\leq 0$ and $\e_2\geq 0$.  

Transition maps between the two charts are
\be r_2=\rrrrr \WWWWW^{1/2}, \ T_2=\frac{\TTTTT}{\WWWWW^{3/2}}, \ \e_2=\frac{1}{\WWWWW^{1/2}}, \ee
\be \TTTTT=\frac{T_2}{\e_2^3}, \ \WWWWW=\frac{1}{\e_2^2}, \ \rrrrr=r_2\e_2. \label{eq:1to3} \ee

\subsection{Chart $K_\e$} \label{sec:Keps}
Chart $K_\e$ is referred to as the re-scaling chart.  Equation (\ref{eq:TWtildes}) expressed in the coordinates of chart $K_\e$ are
\begin{eqnarray}
\dot{\TTTTT}&=&-\rrrrr \WWWWW^2\left(1+\rrrrr^3\TTTTT\right)+\frac{1}{2}\rrrrr \WWWWW\left(c^2-\rrrrr^4\WWWWW^2\right) \nonumber \\
\dot{\WWWWW}&=& -\rrrrr \TTTTT-\rrrrr^4\TTTTT^2 \nonumber \\
\dot{\rrrrr}&=& 0 .
\end{eqnarray}
Dividing the right side of the equation by the common factor $\rrrrr$ through a rescaling of the independent variable we obtain the desingularized equations 
\begin{eqnarray}
\frac{d\TTTTT}{dt_1}&=&- \WWWWW^2\left(1+\rrrrr^3\TTTTT\right)+\frac{1}{2} \WWWWW\left(c^2-\rrrrr^4\WWWWW^2\right) \nonumber \\
\frac{d\WWWWW}{dt_1}&=& - \TTTTT-\rrrrr^3\TTTTT^2 \nonumber \\
\frac{d\rrrrr}{dt_1}&=& 0 .
\end{eqnarray}
The invariant subspace $\rrrrr=0$ corresponds to the flow on the surface of the sphere in blown-up coordinates where
\begin{eqnarray}
\frac{d\TTTTT}{dt_1}&=&- \WWWWW^2+\frac{c^2}{2} \WWWWW \nonumber \\
\frac{d\WWWWW}{dt_1}&=& - \TTTTT. \label{eq:Keonsphere}
\end{eqnarray}
Equation (\ref{eq:Keonsphere}) has fixed points at $(0,0)$ (a center) and at $\left(0,\frac{c^2}{2}\right)$ (a saddle).  The saddle fixed point in this chart corresponds to the fixed point $(1,w_-(0))$ while the fixed point at the origin is an artifact of the desingularization that removed the singularity at $(1,c)$ in the original coordinates. Thus, the blow-up technique is able to differentiate between these two fixed points in the limit as $\e\to 0$ and the fixed point at $(1,w_-(0))$ becomes hyperbolic when viewed in this chart. 

System (\ref{eq:Keonsphere}) is also Hamiltonian with
\[ H(\TTTTT,\WWWWW)=\frac{\WWWWW^3}{3}-\frac{c^2}{4} \WWWWW^2-\frac{\TTTTT^2}{2}. \]
The unstable manifold of $\left(0,\frac{c^2}{2}\right)$ is contained in the level set of $H(\TTTTT,\WWWWW)=-\frac{1}{48}c^6$.  This level curve is plotted in Figure~\ref{fig:hamiltonian} for $c= \sqrt[3]{\frac{3}{2}}$.  

\begin{figure} 
    \centering
     \subfigure{\includegraphics[width=0.43\textwidth]{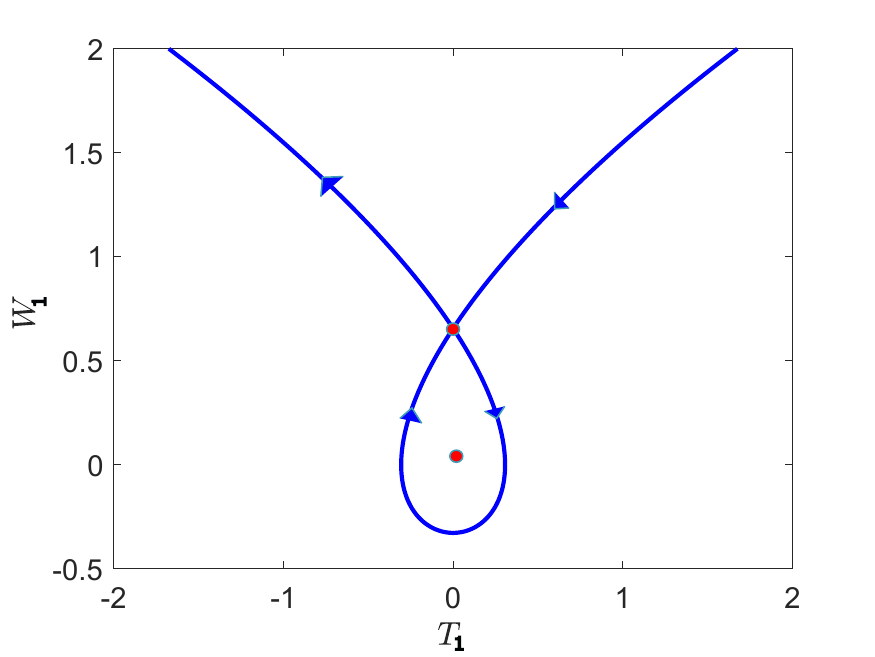}}
 \caption{The dynamics in chart $K_\e$ for $c= \sqrt[3]{\frac{3}{2}}$.  The system is Hamiltonian and there exists an explicit expression for the graph of the unstable manifold that exists for $\TTTTT<0$.   The unstable manifold in upper left quadrant will end up defining the singular orbit in the blow-up system. } 
    \label{fig:hamiltonian}
\end{figure}

\subsection{Chart $K_W$} 
We now consider the dynamics in the chart $K_W$.  Equation (\ref{eq:TWtildes}) transformed to the coordinates of this chart are given by
\begin{eqnarray}
\dot{r_2}&=& -\frac{1}{2}r_2^2T_2-\frac{1}{2}r_2^5T_2^2 \nonumber \\
\dot{\e_2}&=& \frac{1}{2}r_2\e_2T_2+\frac{1}{2}r_2^4\e_2T_2^2 \nonumber \\
\dot{T_2}&=& -r_2(1+r_2^3T_2)+\frac{1}{2}r_2 \e_2^2 \left(c^2-r_2^4\right)+\frac{3}{2}r_2T_2^2+\frac{3}{2}r_2^4T_2^3 . \label{eq:chart1first}
\end{eqnarray}
The vector field can be desingularized by dividing the right side by the common factor $r_2$ after which (\ref{eq:chart1first}) is transformed to 
\begin{eqnarray}
\frac{dr_2}{dt_2}&=& -\frac{1}{2}r_2T_2-\frac{1}{2}r_2^4T_2^2 \nonumber \\
\frac{d\e_2}{dt_2}&=& \frac{1}{2}\e_2T_2+\frac{1}{2}r_2^3\e_2T_2^2 \nonumber \\
\frac{dT_2}{dt_2}&=& -(1+r_2^3T_2)+\frac{1}{2} \e_2^2 \left(c^2-r_2^4\right)+\frac{3}{2}T_2^2+\frac{3}{2}r_2^3T_2^3 . \label{eq:1desing}
\end{eqnarray}
\paragraph{The invariant subspace $r_2=0$}
Restricting to the invariant subspace $r_2=0$ we obtain the system
\begin{eqnarray}
\frac{d\e_2}{dt_2}&=& \frac{1}{2}\e_2T_2. \nonumber \\
\frac{dT_2}{dt_2}&=& -1+\frac{1}{2} \e_2^2c^2+\frac{3}{2}T_2^2. 
\end{eqnarray}
This system has four fixed points: $\left(0,\pm\sqrt{\frac{2}{3}}\right)$ and $\left(\pm\sqrt{\frac{2}{c^2}},0\right)$.  Using $(\ref{eq:1to3})$ we see that the fixed point $\left(\sqrt{\frac{2}{c^2}},0\right)$ is equivalent to $(1,w_-(0))$ in the original coordinates and $\left(\frac{c^2}{2},0\right)$ in the chart $K_\e$.  The fixed point $(0,-\sqrt{2/3})$ is stable with eigenvalues $-3\sqrt{2/3}$ and $-\frac{1}{2}\sqrt{2/3}$ (note the $6:1$ resonance).

\paragraph{The invariant subspace $\e_2=0$} 
The dynamics within the invariant $\e_2=0$ subspace  in chart $K_W$ are given by 
\begin{eqnarray*}
\dot{r_2} &=& -\frac{1}{2} r_2^2T_2 (1+r_2^3T_2) \\
\dot{T_2} &=& r_2\left(-1+\frac{3}{2}T_2^2\right)(1+r_2^3T_2).
\end{eqnarray*}
  Near the fixed point $(r_2,T_2)=(0,-\sqrt{2/3})$ we can divide the right hand side of the vector field by the quantity $r_2(1+r_2^3T_2)$ and obtain the system
\begin{eqnarray*}
\frac{dr_2}{dt_3} &=& -\frac{1}{2} r_2T_2  \\
\frac{dT_2}{dt_3} &=& \left(-1+\frac{3}{2}T_2^2\right).
\end{eqnarray*}
The fixed point $(r_2,T_2)=(0,-\sqrt{2/3})$ is a saddle and the graph of the unstable manifold is simply  the line $T_2=-\sqrt{2/3}$.

Converting this back to the coordinates $(T,W)$ this manifold is expressed as the graph
\[ T=1-\sqrt{\frac{2}{3}}r_2^3, \quad W=c-r_2^2.\]
Using these identities we then find an implicit equation relating $T$ and $W$
\be (c-W)^3=\frac{3}{2}(1-T)^2.\label{eq:implicitagain} \ee
This is exactly the same implicit relation given in (\ref{eq:solutioncurves}) with $k=-\frac{1}{2}$ implying that the unstable manifold of $(0,0,-\sqrt{2/3})$, when transformed back to the original variables $(T,W)$ is described by the graph  $W=g_u\left(c,T\right)$.

\paragraph{The transition map}

To analyze the transition map for (\ref{eq:1desing}) we divide the vector field by the (locally) positive factor $\frac{1}{2}(-T_2-r_2^3T_2^2)$ effectively introducing a new independent variable, $\sigma$. We then shift the fixed point to the origin $S_2=T_2+\sqrt{\frac{2}{3}}$ after which (\ref{eq:1desing}) takes the form 
\begin{eqnarray}
\frac{dr_2}{d\sigma}&=& r_2 \nonumber \\
\frac{d\e_2}{d\sigma}&=& -\e_2 \nonumber \\
\frac{dS_2}{d\sigma}&=&-6S_2+ S_2 G_1(r_2,S_2)+\e_2 G_2(r_2,\e_2,S_2), \label{eq:1desing2}
\end{eqnarray}
Here $G_1$ and $G_2$ are nonlinear functions that represent the nonlinearity after the change of coordinates.  In these coordinates the unstable manifold of the origin is the $r_2$ axis, so that there are no nonlinear terms in the equation for $S_2$ that depend only on $r_2$.  We further have that the function $G_1(0,0)=0$ since it contains only nonlinear terms and $G_2(0,0,0)=0$ since the linear term $\e_2$ is absent from (\ref{eq:1desing2}).


For $\kappa>0$ define the sections $\Sigma_{in}$ and $\Sigma_{out}$ as 
\[ \Sigma_{in}=\{ (r_2,\e_2,S_2)\ | \ \e_2=\kappa\}, \  \Sigma_{out}=\{ (r_2,\e_2,S_2)\ | \ r_2=\kappa\}.\]
We study the transition map 
\[ \pi:\Sigma_{in}\to\Sigma_{out},\ \quad (r_2(0),\kappa,S_2(0))\to (\kappa,\e_2(\sigma_{out}),S_2(\sigma_{out})), \] 
for some transition time $\sigma_{out}$.  
Analysis of the map is possible using a Shilnikov type analysis; see for example \cite{deng89,shilnikov70}.  We have the following result. 

\begin{lemma}\label{lem:pi} Consider (\ref{eq:1desing2}) with $r_2(0)=\e\Gamma(c,\e;\kappa)$ and $S_2(0)=\Omega(c,\e;\kappa)$ for some functions $\Gamma(c,\e;\kappa)$ and $\Omega(c,\e;\kappa)$, both smooth in $c$ and $\e$.  Then if $\kappa$ is sufficiently small and  $|\Omega(c,\e;\kappa)|<\frac{\kappa}{2}$ then  it holds that 
\[ \pi(\e\Gamma(c,\e;\kappa),\kappa,\Omega(c,\e;\kappa))=\left(\kappa, \e\Gamma(c,\e;\kappa),\e \Xi(c,\e;\kappa)\right) \]
for some continuously differentiable function $\Xi(c,\e;\kappa)$.
\end{lemma}

\begin{rmk} The fixed point in (\ref{eq:1desing2}) is hyperbolic and so we could linearize the system by a near-identity change of coordinates using Hartman-Grobman Theorem.  Resonances between the eigenvalues limit the smoothness of the conjugacy mapping and, even still, the mere existence of a near-identity conjugacy is not sufficient to obtain the estimates that we desire.  
\end{rmk}
\begin{proof}
Proofs for more general Shilnikov type problems of this form can be found in \cite{deng89,shilnikov70}.  We present some elements of the proof here.  
Take the integral form of (\ref{eq:1desing2}),
\begin{eqnarray} S_2(\sigma)&=&e^{-6\sigma}S_2(0)+\int_0^\sigma e^{-6(\sigma-\tau)}S_2(\tau) G_1\left(\e\Gamma(c,\e)e^\tau,S_2(\tau)\right)\mathrm{d}\tau \nonumber \\
&+&\kappa \int_0^\sigma e^{-6(\sigma-\tau)} e^{-\tau} G_2 \left(\e\Gamma(c,\e)e^\tau,\kappa e^{-\tau}, S_2(\tau)\right)\mathrm{d}\tau. \label{eq:Qdef}
\end{eqnarray}
The transition ``time" required for the solution to pass from $\Sigma_{in}$ to $\Sigma_{out}$ is 
\[ \sigma^*=\log\left(\frac{\kappa}{\e\Gamma(c,\e;\kappa)}\right).\]
We need estimates on the solution of this equation on the interval $[0,\sigma^*]$. Define $\Psi:X\to X$ as the right hand side of (\ref{eq:Qdef}) and let $X$ be the  Banach Space 
\[ X=C^0([0,\sigma^*],[-\kappa,\kappa]), \ ||S||_X=\sup_{0<\sigma<\sigma^*} |e^{\sigma}S(\sigma)|. \] 
Then for any $S\in X$ we have
\begin{eqnarray} \left| e^\sigma \Psi Q\right|&\leq& \left| e^{-5\sigma}Q(0)\right|+ \int_0^\sigma e^{-5(\sigma-\tau)} ||S||_X |G_1(\tau)| \mathrm{d} \tau\\
&+& \kappa \int_0^\sigma e^{-5(\sigma-\tau)} |G_2(\tau)| \mathrm{d}\tau.
\end{eqnarray}
Thus 
\[ ||\Psi S||_X\leq |\Omega(c,\e;\kappa)| +C_1 \kappa ||S||_X+C_2 \kappa^2, \] 
and there is a $\kappa$ sufficiently small such that $\Psi:X\to X$.  
Next we show that $\Psi$ is a contraction on $X$.  Consider $S_a$ and $S_b$, both in $X$.  Then 
\[ \left| e^\sigma \Psi (S_a-S_b)\right| \leq L_1\kappa  \int_0^\sigma e^{-5(\sigma-\tau)} ||S_a-S_b||_X \mathrm{d} \tau + \kappa L_2 \int_0^\sigma e^{-5(\sigma-\tau)} ||S_a-S_b||_X \mathrm{d}\tau,
\]
where $L_1$ is the Lipschitz constant of $G_1$ and $L_2$ is the Lipschitz constant for $G_2(r_2,\e_2,S_2)-G_\e(r_2,\e_2,0)$.  We then have 
\[ ||\Psi(S_a-S_b)||_X\leq C\kappa ||S_a-S_b||_X \]
and $\Psi$ is a contraction for $\kappa$ sufficiently small. 

Now consider the $S_2$ coordinate of the transition map $\pi$.  Since $S_2\in X$ we have that 
\[ |S_2(\sigma^*)|\leq e^{-\sigma^*}||S_2||_X \]
 and thus $S_2(\sigma^*)=\e\Xi(c,\e;\kappa)$ for some function $\Xi(c,\e;\kappa)$. By Theorem 8.1 of \cite{deng89} the solution of the boundary value problem (\ref{eq:1desing2}) depends smoothly on the parameters $c$ and $\e$ as well as the initial conditions.

\end{proof}

\subsection{Tracking $W^u(1,w_-(\e))$ between charts} 
We now put together the analysis in the two charts above to track $W^u(1,w_-(\e))$.  In the re-scaling chart $K_\e$ the equilibrium point $(1,w_-(\e))$ is mapped to (for $\e=0$ ) the fixed point $\left(0,\frac{c^2}{2}\right)$.  This fixed point is hyperbolic of saddle type and has a one-dimensional unstable manifold.  The fixed point and its unstable manifold persist for $\e>0$ and sufficiently small.  Using the Hamiltonian function in chart $K_\e$ we find that the unstable manifold (when $\e=0$) is expressed as the graph
\[ \TTTTT=-\sqrt{\frac{2}{3}\WWWWW^3-\frac{c^2}{4} \WWWWW^2+\frac{1}{24}c^6}.\]
By smooth dependence on initial conditions and parameters the unstable manifold for $\e\neq 0$ can be described by the graph
\[ \TTTTT=-\sqrt{\frac{2}{3}\WWWWW^3-\frac{c^2}{4} \WWWWW^2+\frac{1}{24}c^6}+\e \Delta(\WWWWW,c,\e),\]
for some smooth function $\Delta$.  

When $\WWWWW=\frac{1}{\kappa^2}$ then 
\[ \TTTTT=-\sqrt{\frac{2}{3}\frac{1}{\kappa^6}-\frac{c^2}{4} \frac{1}{\kappa^4}+\frac{1}{24}c^6}+\e\Delta\left(\frac{1}{\kappa^2},c,\e\right). \]
Transforming to chart $K_W$ using (\ref{eq:1to3}) then we have $\e_2=\kappa$, $r_2=\e\kappa$ and 
\[ T_2=-\sqrt{\frac{2}{3}-\frac{c^2}{4} \kappa^2 +\frac{1}{24}(\kappa c)^6}+\e\kappa^3 \Delta\left(\frac{1}{\kappa^2},c,\e\right).\]
Then
\[ S_2=\sqrt{\frac{2}{3}}-\sqrt{\frac{2}{3}-\frac{c^2}{4} \kappa^2 +\frac{1}{24}(\kappa c)^6}+\e\kappa^3 \Delta\left(\frac{1}{\kappa^2},c,\e\right), \]
and we observe that $S_2=\Omega(c,\e,\kappa)$ with
\[ \Omega(c,\e,\kappa)=\sqrt{\frac{1}{6}}\frac{c^2}{4}\kappa^2+\kappa^3 \Pi(\kappa,c,\e), \]  
for some smooth function $\Pi$.  

Now applying the transition map $\pi:\Sigma_{in}\to\Sigma_{out}$ in Lemma~\ref{lem:pi} we can track $W^u(1,w_-(\e))$ to $\Sigma_{out}$ where it has the  expansion
\[ r_2=\kappa, \ T_2=-\sqrt{\frac{2}{3}}+\e\Xi(c,\e), \ \e_2=\e \Gamma(c,\e). \]

\begin{figure} 
    \centering
     \subfigure{\includegraphics[width=0.43\textwidth]{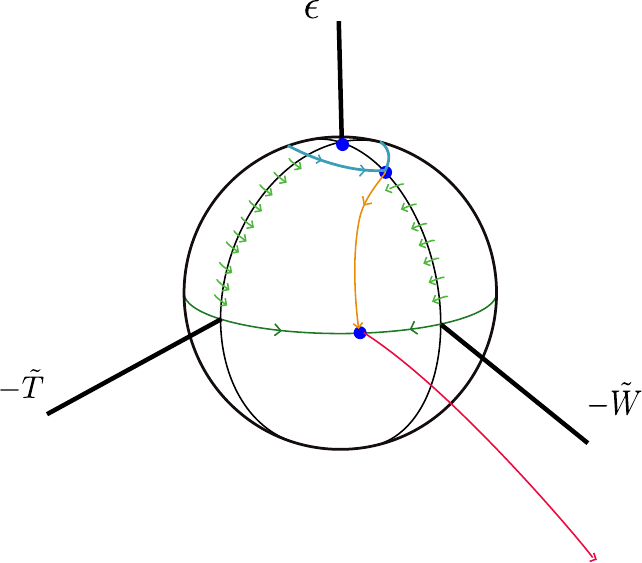}}
 \caption{ The relevant dynamics of the blown up fixed point at $(1,c)$.  The $\e=0$ subspace is invariant and we have shown the existence of a heteroclinic connecting this fixed point to the origin in Section~\ref{sec:hetero}.  This is depicted by the red curve.  On the upper half of the sphere two fixed points are shown.  The one at the north pole is an artificial fixed point that arises when the singularity occurring at $W=c$ in (\ref{eq:scale1}) is removed.  The other fixed point corresponds to the steady state $(1,w_-(0))$.  In the blown-up coordinates this fixed point is hyperbolic and we are able to track its unstable manifold and show that it is heteroclinic to the fixed point at the equator of the sphere.  Obtaining estimates on the dynamics for $0<\e\ll 1$ we are able to show that the unstable manifold of $(1,w_-(\e))$ is $\O(\e)$ close to the reduced heteroclinic shown in red.   }
    \label{fig:sphere}
\end{figure}

\begin{rmk} In theory, all of the analysis required in this section could have been conducted solely in chart $K_W$ as both the relevant fixed points are visible in that chart.  However, identifying the heteroclinic connecting the two fixed points is more straightforward in the re-scaling chart $K_\e$ and so we opt to utilize both charts here. 
\end{rmk}

\section{Proof of Theorem~\ref{thm:main}}\label{sec:proof}

In this section we prove Theorem~\ref{thm:main} by verifying that the conditions of the Implicit Function Theorem outlined in Theorem~\ref{thm:Phifacts} are satisfied. 

By the analysis in the previous section we have that the unstable manifold of $W^u(1,w_-(\e))$ is well defined and its graph is $\O(\e)$ close to $(T,g_u(c,T))$;  see the discussion following (\ref{eq:implicitagain}).  On the other hand, $W^{ss}(0,0)$ is well defined and depends smoothly on $c$ and $\e$ and for any $0<T<1$ is $\O(\e)$ close to $(T,g_s(c,T))$.  Thus, recalling the definition of $\Phi(c,\e)$ in (\ref{eq:Phi}) we have
\[ \Phi(c,\e)=g_s\left(c,\frac{1}{2}\right)-g_u\left(c,\frac{1}{2}\right)+\e R(c,\e) \]
where $R(c,\e)$ is smooth.  

When $\e=0$ and $c^*=\sqrt[3]{\frac{3}{2}}$ then by the analysis presented in Section~\ref{sec:hetero} we have that $\Phi(c^*,0)=0$ and furthermore by Lemma~\ref{lem:IFTcond} we have that $\partial_c \Phi(c^*,0)\neq 0$.  

Thus $\Phi(c,\e)$ satisfies the three conditions set out in Theorem~\ref{thm:Phifacts} necessary for an application of the Implicit Function Theorem.  This gives the existence of a smooth function $c^*(\e)$ such that (\ref{eq:RDA}) with $\nu=0$ has a traveling front solution with steep exponential decay that propagates with speed $c=c^*(\e)$ and for $\rho=\frac{1}{\e^3}$.  

Monotonicity of the front was previously established in \cite{bramburger21}.  The decay estimate in Theorem~\ref{thm:main} follows since the front lies in $W^{ss}(0,0)$.  This decay rate is specified in (\ref{eq:mu}) and after twice re-scaling the independent variable -- first to divide the right hand side of the differential equation by a factor of $c-W$ and second to re-scale the independent variable back to $x$ -- then the estimate is obtained.  

\section{Discussion}\label{sec:discussion}

We have constructed traveling front solutions of the system of reaction-diffusion-advection equations in (\ref{eq:RDA}) for $\nu=0$ and in the limit $\rho\to\infty$.   Our approach was to construct fronts directly using singular perturbation techniques.  In the original traveling wave coordinates the profile of the $u$ component becomes asymptotically large necessitating  a rescaling of this dependent variable.  This rescaling introduces a lack of hyperbolicity into the problem which can be handled using geometric desingularization techniques.  

Our results confirm a conjecture presented in \cite{bramburger21} regarding the scaling of the critical speed as $\rho\to\infty$ in the inviscid case; $\nu=0$.  We comment briefly on extensions to the viscid case.  The traveling wave equations in this case naturally comprise a system of four coupled ordinary differential equations.  This system again has a conserved quantity (see Section 4 of \cite{bramburger21}) and therefore the system reduces to a three-dimensional differential equation which we write as follows:
\begin{eqnarray}
\dot{T}&=& T (W-c) +\e^2 Z \nonumber \\
\nu \dot{W} &=& -\frac{W}{2}\left(2c-W\right)+Z \nonumber \\
\dot{Z}&=& T(T-1). \label{eq:nunonzero}
\end{eqnarray}
To see the connection with the reduced system  (\ref{eq:scale1}) note that when $\nu=0$ we have $Z=\frac{W}{2}\left(2c-W\right)$ and $\dot{Z}=(c-W)\dot{W}$.   We first consider the case of $\e$ (equivalently $\rho$) fixed and $\nu\ll 1$.  In this regime, $Z=\frac{W}{2}\left(2c-W\right)$ is a slow manifold which is normally hyperbolic when $W<c$.  By Fenichel's persistence theorem; see \cite{fenichel79}, the manifold will persist for $\nu>0$ and sufficiently small.  Therefore, we would expect that any pushed or pulled fronts constructed in the $\nu=0$ limit (and for finite $\rho$) should perturb smoothly to pushed or pulled fronts in the case of $\nu>0$.   Note that this does not include the case of $\rho\to \infty$ as the fronts will limit on the non-hyperbolic point of the slow manifold at $W=c$ in this case.  

Beyond the regime of $\nu$ small the challenge in the analysis of (\ref{eq:nunonzero}) lies in the construction of fronts in the reduced system $\e=0$.  Recall that the planar system obtained when $\e=0$ in (\ref{eq:reduced}) could be simplified to a scalar equation with explicit solution.  We have not been able to obtain a similar explicit solution in the case of $\nu>0$.  It is interesting to note that for $\nu>0$ the linearization of (\ref{eq:nunonzero}) at the reduced fixed point $(1,c,c^2/2)$ is 
\[ Df\left(1,c,\frac{c^2}{2}\right)=\left(\begin{array}{ccc} 0 & 1 & 0 \\ 0 & 0& \frac{1}{\nu}\\1 & 0 & 0 \end{array}\right),\]
the eigenvalues of which are the cube roots of $\frac{1}{\nu}$ and therefore the fixed point is hyperbolic for $\nu>0$ with a one-dimensional unstable manifold.  The fixed point at the origin has two (strong)-stable eigenvalues. Therefore, while one piece of the analysis of fronts for $\nu>0$ is significantly more difficult (the construction of the reduced heteroclinic), the persistence of this front for $\nu>0$ is more straightforward.  

We have referred to the constructed front as a pushed front throughout our analysis.  Returning to the discussion of what it means for a front to be selected in the introduction we note that the constructed front satisfies both the slowest monotone/positive front criteria (at least locally in the wavespeed parameter $c$).  Another direction for possible future research would be into the stability of the front constructed in Theorem~\ref{thm:main}.  We anticipate that this front is asympotically stable in an exponentially weighted function space with an isolated marginally stable eigenvalue at zero due to translational invariance.  That this translational eigenvalue has steep enough decay to belong to the exponentially weighted space was the primary motivation for seeking fronts that lie in the strong-stable manifold of the origin.  The spectral stability problem can be studied via a system of three coupled non-autonomous linear differential equations depending both the front profile and the spectral parameter $\lambda$.  In contrast to the existence problem, we do not expect this equation to have a conserved quantity so that a fully three dimensional analysis would be required.

\section*{Acknowledgements} This project was conducted as part of a semester-long undergraduate research project hosted by the Mason
Experimental Geometry Lab (MEGL). 

\section*{Funding} The research of M.H. was partially supported by the National Science Foundation (DMS-2007759). 

\section*{Author Contribution} All authors contributed to the research. M.H. wrote the manuscript.

\section*{Conflict of Interest} The authors declare that they have no competing interests.

\bibliographystyle{abbrv}
\bibliography{GSPTbib}

\end{document}